\documentclass[11pt]{article}
\usepackage{amsmath,amssymb,amsthm}
\usepackage{tikz}
\usetikzlibrary{calc}
\usepackage{float}
\usepackage{xcolor}
\usepackage{authblk}
\usepackage{autobreak}
\usepackage{enumerate}
\usepackage{hyperref}
\hypersetup{
	colorlinks,
	linkcolor={red!100!black},
	citecolor={green!60!black},
	urlcolor={blue!60!black}
}

\newtheorem{theorem}{Theorem}
\newtheorem{case}{Case}
\newtheorem{lemma}{Lemma}
\newtheorem{corollary}{Corollary}
\theoremstyle{definition}

\newtheorem{claim}{Claim}[section]

\begin{document}

\title{Trichotomy and $tK_m$-goodness of sparse graphs}
\author[1,3]{Yanbo Zhang}
\affil[1]{School of Mathematical Sciences\\ Hebei Normal University\\ Shijiazhuang 050024, China}
\author[2]{Yaojun Chen}
\affil[2]{School of Mathematics\\ Nanjing University\\ Nanjing 210093, China}
\affil[3]{Hebei Research Center of the Basic Discipline Pure Mathematics\\ Shijiazhuang 050024, China}

\date{}
\maketitle
\let\thefootnote\relax\footnotetext{\emph{Email addresses:} {\tt ybzhang@hebtu.edu.cn} (Y. Zhang), {\tt yaojunc@nju.edu.cn} (Y. Chen)}
\begin{quote}
{\bf Abstract:} 
Let $G$ be a connected graph with $n$ vertices and $n+k-2$ edges and $tK_m$ denote the disjoint union of $t$ complete graphs $K_m$. In this paper, by developing a trichotomy  for sparse graphs, we show that for given integers $m\ge 2$ and $t\ge 1$, there exists a positive constant $c$ such that if $1\le k\le cn^{\frac{2}{m-1}}$ and $n$ is large, then $G$ is $tK_m$-good, that is, the Ramsey number is
\[
r(G, tK_m)=(n-1)(m-1)+t\,.
\]
In particular, the above equality holds for any positive integers $k$, $m$, and $t$, provided $n$ is large. The  case $t=1$ was obtained by Burr, Erd\H{o}s, Faudree, Rousseau, and Schelp (1980), and the case $k=1$ was established by  Luo and Peng (2023).


\end{quote}

\section{Introduction}
Given a positive integer $n$ and a graph $G$, how can we ensure that $G$ contains all trees with $n$ vertices? Except for the well-known folklore result that large minimum degree of $G$ can guarantee this, a very useful tool for doing this is ``Tree Trichotomy"  developed by Burr, Erd\H{o}s, Faudree, Rousseau, and Schelp~\cite{Burr1982}, who showed that any tree must contain a long suspended path, i.e., a path in which all internal vertices have degree $2$, or a large matching consisting of end-edges, or a large star consisting of end-edges, where an end-edge is one incident with a vertex of degree 1. To be precise, they show the following result, referred to as ``Tree Trichotomy'' by Faudree, Rousseau, and Schelp~\cite{Faudree1990}.

\begin{lemma}[Burr and Faudree~\cite{Burr1993}]
	If $n\ge 4\alpha \beta \gamma$, then every tree with $n$ vertices contains a suspended path with $\alpha$ vertices, or a matching consisting of $\beta$ end-edges, or a star consisting of $\gamma$ end-edges.
\end{lemma}

This method was subsequently used multiple times~\cite{Burr1987, Burr1993, Erdos1989, Erdos1982, Erdos1985, Erdos1988, FaudreeSS1990} between 1982 and 1993. The core idea of this method is to use the Tree Trichotomy Lemma to divide the problem into three cases. In each case, we use induction to identify a smaller tree $T'$. Additionally, in the third case, a greedy algorithm is employed. The structure of $T'$ depends on the specific case: it is obtained by shortening a suspended path (first case), deleting some end-edges forming a matching (second case), or deleting some end-edges forming a star (third case).

Next, other tools are utilized to extend $T'$ into the desired $n$-vertex tree $T_n$. For the second case, Hall's theorem is applied to ensure the embedding of a matching. In the third case, it is necessary to first identify a star $K_{1,n-1}$. The first case relies on the following path extension lemma established by Erd\H{o}s, Faudree, Rousseau, and Schelp~\cite{Erdos1985}, which is frequently used in constructing long paths in sparse graphs.

\begin{lemma}[Erd\H{o}s, Faudree, Rousseau, and Schelp~\cite{Erdos1985}]\label{pathextension}
	Let $K_{a+b}$ be a complete graph on the vertex set $\{x_1, \dots, x_a, y_1, \dots, y_b\}$, with edges colored red or blue. Assume that there is a red path $x_1x_2 \cdots x_a$ of length $a-1$ joining $x_1$ and $x_a$. If $a\ge b(c-1)+d$, then at least one of the following holds:
	\begin{enumerate}[(1)]
		\item A red path of length $a$ connects $x_1$ to $x_a$.
		\item A blue complete subgraph $K_c$ exists.
		\item There are $d$ vertices in $\{x_1, x_2, \dots, x_a\}$ that are joined in blue to every vertex in $\{y_1, y_2, \dots, y_b\}$.
	\end{enumerate}	
\end{lemma}

Now, let us explore how the above two lemmas can be applied to the Ramsey numbers involving trees. Given two graphs $G$ and $H$, their \emph{Ramsey number} $r(G, H)$ is defined as the smallest positive integer $N$ such that, for every red-blue edge coloring of the complete graph $K_N$, there exists either a red subgraph isomorphic to $G$ or a blue subgraph isomorphic to $H$. Burr~\cite{Burr1981} observed the following result. Prior to this, a weaker version, which did not involve the concept of chromatic surplus, was obtained by Chv\'atal and Harary~\cite{Chvatal1972}. The \emph{chromatic surplus} of a graph is defined as the cardinality of the smallest color class, taken over all proper colorings of the graph that use the minimum number of colors.

\begin{lemma}[Burr~\cite{Burr1981}] \label{lowerbound}
    Let $G$ be a connected graph with $n$ vertices, and let the chromatic number and chromatic surplus of $H$ be denoted by $\chi$ and $s$, respectively. If $n\ge s$, then 
    \begin{equation}\label{lower}
        r(G, H)\ge (n-1)(\chi-1)+s\,.
    \end{equation}
\end{lemma}

To establish a unified formula for the Ramsey numbers of a broader class of graphs, Burr~\cite{Burr1981} introduced the notion of Ramsey goodness. A connected graph $G$ is said to be \emph{$H$-good} if equality holds in inequality~(\ref{lower}). In the special case where $H=K_k$, the graph $G$ is referred to as \emph{$k$-good} by Burr and Erdős~\cite{Burr1983}.
One of the most classical results in this area is due to Chv\'atal~\cite{Chvatal1977}, who proved that all trees $T_n$ are $k$-good. 
\begin{theorem}[Chv\'atal~\cite{Chvatal1977}]\label{Chv}
	$r(T_n,K_m)=(n-1)(m-1)+1\,.$
\end{theorem}

\noindent Since then, various generalizations and analogs of this result have been developed, establishing its central role in graph Ramsey theory.  Perhaps motivated by Theorem \ref{Chv}, 
Burr, Erd\H{o}s, Faudree, Rousseau, and Schelp~\cite{Burr1980}  started to consider when a connected sparse graph $G$ of size $n+k-2$ is $K_m$-good and obtained the following.

\begin{theorem}[Burr, Erd\H{o}s, Faudree, Rousseau, and Schelp~\cite{Burr1980}]\label{BEFRS}
	Let $G$ be a connected graph with $n$ vertices and $n+k-2$ edges. For $m\ge 3$, there exists a positive constant $\varepsilon$ such that if $1\le k\le \varepsilon n^{\frac{2}{m-1}}$ and $n$ is sufficiently large, then
	\[
	r(G,K_m)=(n-1)(m-1)+1\,.\]
\end{theorem}

If $H$ is not a complete graph but rather the disjoint union of copies of a complete graph, an intriguing question arises: are trees $T_n$ still $H$-good? This problem has only recently begun to attract attention. In 2023, the following three results were published: Hu and Peng~\cite{Hu1} demonstrated that any tree $T_n$ is $2K_m$-good for $n\ge 3$ and $m\ge 2$. They~\cite{Hu2} also established that any large tree is $tK_2$-good for $t\ge 3$. Subsequently, Luo and Peng~\cite{Luo2023} extended these two results, proving that any large tree is $tK_m$-good.

\begin{theorem}[Luo and Peng~\cite{Luo2023}]\label{LuoPeng}
	For any $t\ge 1$ and $m\ge 2$, if $n$ is sufficiently large, then 
\[
r(T_n,tK_m)=(n-1)(m-1)+t\,.
\]
\end{theorem}

The proof of Theorem \ref{LuoPeng} relies on the method of Tree Trichotomy. To the best of our knowledge, this method remained dormant for thirty years until it was recently revived by the above theorem.

Note that a tree $T_n$ is a minimal connected graph with $n$ vertices and $n-1$ edges, and it is clearly a sparse graph. A  natural question is:  Which sparse graphs are $tK_m$-good?
In particular, give a good bound for which any connected sparse graph of sufficiently large order $n$ with size not exceeding the bound is $tK_m$-good?

In this paper, we investigate the upper bound and the main result  is as follows. 

\begin{theorem}\label{main}
	Let $G$ be a connected graph with $n$ vertices and $n+k-2$ edges. For any integers $m\ge 2$ and $t\ge 1$, there exists a positive constant $c$ such that if $1\le k\le cn^{\frac{2}{m-1}}$ and $n$ is sufficiently large, then
	\[
		r(G,tK_m)=(n-1)(m-1)+t\,.\]
\end{theorem}

Obviously, Theorem~\ref{BEFRS} is the special case ($t=1$) of Theorem \ref{main}. Note that for any positive integers $k$, $t$, and $m$ with $m\ge 2$, provided that $n$ is sufficiently large, the condition $1\le k\le cn^{\frac{2}{m-1}}$ always holds. Consequently, we immediately derive the following corollary, where the case $m=1$ holds trivially.

\begin{corollary}
	Let $k$, $m$, and $t$ be positive integers. For any connected graph $G$ with $n$ vertices and $n+k-2$ edges, if $n$ is sufficiently large, then \[r(G, tK_m)=(n-1)(m-1)+t\,.\]
\end{corollary}

Thus, Theorem~\ref{LuoPeng} can be regarded as a special case of Corollary 1 when $k=1$. Compared to the proof of Theorem~\ref{LuoPeng}, the establishment of our main result encounters two major challenges: first, in the third case of the trichotomy method, sparse graphs cannot be embedded directly using a greedy algorithm, as is possible for trees; second, we need to establish an upper bound for $k$ expressed as a polynomial function of $n$. To overcome these challenges, we develop an enhanced version of the Tree Trichotomy Lemma for sparse graphs, which is the key ingredient for proving Theorem \ref{main} and holds potential for future applications. 

\begin{lemma}\label{Gpsize}
	Let $G$ be a connected graph with $n$ vertices and $n+k-2$ edges, where $n\ge q\ge 3$ and $k\ge 1$. If $G$ contains neither a suspended path of order $q$ nor a matching consisting of $\ell$ end-edges, then the number of vertices of degree at least $2$ in $G$ is at most $\alpha$, and $G$ contains a star with at least $\left\lceil \frac{n-\alpha}{\ell-1} \right\rceil$ end-edges, where $\alpha=(q-2)(2\ell+3k-8)+1$.
\end{lemma}

The rest of this paper is organized as follows.  Because the proof of Theorem~\ref{main} requires the upper bound for the Ramsey number of a graph $G$ versus a $K_m$, 
in Section~\ref{section2}, we provide a new upper bound for this Ramsey number as a by-product, which  refines the result of Burr, Erd\H{o}s, Faudree, Rousseau, and Schelp~\cite{Burr1980}. Section~\ref{section3} is devoted to proving the key Lemma \ref{Gpsize}. Section~\ref{section4} provides additional lemmas needed for the proofs. Finally, we present the proof of the main theorem.

We conclude this section by introducing some additional notation. Tur\'an number $\operatorname{ex}(n, H)$ is  the maximum number of edges in an $n$-vertex simple graph that does not contain $H$ as a subgraph.
For a graph $G$, let $|G|$ and $e(G)$ denote the number of vertices and edges in $G$, respectively. The graph $G+H$ represents the disjoint union of two graphs $G$ and $H$, where every vertex of $G$ is adjacent to every vertex of $H$. The red neighborhood of a vertex $v$ is defined as the set of vertices that are adjacent to $v$ via red edges. For a complete graph $K_N$ and a vertex subset $A$, we use $K_N[A]$ to denote the subgraph induced by $A$.

\section{A new upper bound for $r(G,K_m)$}\label{section2}

When $m\ge 3$, for a graph $G$ with $n$ vertices and $\ell$ edges, Burr, Erd\H{o}s, Faudree, Rousseau, and Schelp~\cite{Burr1980} established that $r(G,K_m)\le(n+2\ell)^{\frac{m-1}{2}}$. Although this result is sufficient for our use, it is still meaningful to consider improvements to the base of the bound. Here, we present a strengthened result as follows.

\begin{lemma}\label{upper}
		Let $G$ be a graph with $\ell$ edges and no isolated vertices. For $m\ge 3$, we have
		\[
		r(G,K_m)\le (2\ell+1)^{\frac{m-1}{2}}.
		\]
	\end{lemma}
Take $m=3$ in Lemma \ref{upper}, we have the following result, which was obtained by two groups of authors, independently, so Lemma \ref{upper} can be regarded as a generalization of it. However, the proof of Lemma \ref{upper}  relies on the  result.

\begin{lemma}[Sidorenko~\cite{Sidorenko1993}; Goddard and Kleitman~\cite{Goddard1994}]\label{triangle}
    For any graph $G$ of $\ell$ edges without isolated vertices,
    \[
    R(G,K_3)\le 2\ell+1\,.
    \]
\end{lemma}

	\begin{proof}[Proof of Lemma~\ref{upper}]
	
		We denote by $n$ the number of vertices in $G$. When the graph $G$ consists of only one edge, we have $r(K_2,K_m)=m\le 3^{\frac{m-1}{2}}$, and the lemma holds. Therefore, assume that $\ell\ge 2$.
		
		We proceed by induction on $m$. When $m=3$, the lemma holds by Lemma \ref{triangle}. Assume that the lemma holds for smaller values of $m$, and now consider the case where $m\ge 4$.
	
		We use the method of minimal counterexample. Let $G$ be a graph with the smallest number of edges for which the lemma does not hold. Let $N=\left\lfloor (2\ell+1)^{\frac{m-1}{2}} \right\rfloor$. Then there exists a red-blue edge coloring of the complete graph $K_N$ such that there is neither a red subgraph isomorphic to $G$ nor a blue subgraph isomorphic to $K_m$.
	
		If $G$ is disconnected, then $G$ can be decomposed into two disjoint subgraphs $G_1$ and $G_2$, where $G_1$ is a connected component of $G$. Let $G_1$ have $\ell_1$ edges. Since $G$ is the minimal counterexample, $K_N$ contains a red subgraph isomorphic to $G_1$. Moreover, in $K_N-V(G_1)$, we have
		\[
		r(G_2, K_m)\le (2(\ell-\ell_1)+1)^{\frac{m-1}{2}}\le (2\ell+1)^{\frac{m-1}{2}}-2\ell_1.
		\]
		The second inequality follows easily from $m\ge 4$. Thus, $K_N$ contains a red subgraph isomorphic to $G_1 \cup G_2$, which is $G$, leading to a contradiction. In the following, we assume that $G$ is connected.
	
		Let $v$ be a vertex in $G$ with the smallest degree, where the degree of $v$ is denoted by $d$. Denote $G-v$ by $H$. Since $G$ is the minimal counterexample, $r(H, K_m)\le N$. Thus, the complete graph $K_N$ contains a red subgraph isomorphic to $H$.
	
		Since the subgraph $H$ can be embedded in the red subgraph of $K_N$, we assume that the neighbors $v_1,v_2, \ldots, v_d$ of vertex $v$ in $G$ are embedded into vertices $x_1,x_2, \ldots,x_d$ in $K_N$, respectively. If there exists a vertex $x$ in $K_N-V(H)$ such that all edges between $x$ and $x_1, x_2, \dots, x_d$ are red, then $K_N$ would contain $G$ as a red subgraph, which contradicts our assumption. Therefore, the set $\{x_1, x_2, \dots, x_d\}$ must have at least $N-(n-1)$ blue edges to vertices in $K_N-V(H)$. We now state the following claim.

	 \begin{claim}\label{claim1}
		$N-(n-1) > d(r(G,K_{m-1})-1)$.
	\end{claim}
	\begin{proof}
		Since $m\ge 4$ and $G$ is a connected graph, we have 
		\[
		\frac{1}{2}(2\ell+1)^{\frac{m}{2}-1}\ge \frac{1}{2}(2\ell+1) > n-\frac{2\ell}{n}.
		\]
		This implies 
		\[
		\frac{1}{2} > \frac{n-\frac{2\ell}{n}}{(2\ell+1)^{\frac{m}{2}-1}}.
		\]
		Observing that $N+1 > (2\ell+1)^{\frac{m-1}{2}}$ and $r(G,K_{m-1}) \le (2\ell+1)^{\frac{m-2}{2}}$, we proceed as follows:
		\begin{align*}
			&1=\frac{2\binom{n}{2}}{n^2}+\frac{1}{n} \\
			\Longrightarrow\quad&1\ge \frac{2\ell}{n^2}+\frac{1}{n}\\
			\Longrightarrow\quad&2\ell\ge \frac{4\ell^2}{n^2}+\frac{2\ell}{n}\\
			\Longrightarrow\quad&2\ell+1>\left(\frac{2\ell}{n}+\frac{1}{2}\right)^2\\
			\Longrightarrow\quad&(2\ell+1)^{\frac{1}{2}}>\frac{2\ell}{n}+\frac{n-\frac{2\ell}{n}}{(2\ell+1)^{\frac{m}{2}-1}}\\
			\Longrightarrow\quad&(2\ell+1)^{\frac{m-1}{2}}>\frac{2\ell}{n}(2\ell+1)^{\frac{m-2}{2}}+\left(n-\frac{2\ell}{n}\right)\\
			\Longrightarrow\quad&
			N-(n-1)>\frac{2\ell}{n}((2\ell+1)^{\frac{m-2}{2}}-1)\\
			\Longrightarrow\quad&
			N-(n-1)>d(r(G,K_{m-1})-1).\\
		\end{align*}
		This completes the proof of the claim.
	\end{proof}

	Let us continue the proof of Lemma~\ref{upper}. Recall that there are at least $N-(n-1)$ blue edges between $\{x_1,x_2,\ldots,x_d\}$ and $K_N-V(H)$. By Claim \ref{claim1} and the pigeonhole principle, there exists at least one vertex, say $x_1$, that has at least $r(G,K_{m-1})$ blue edges to $K_N-V(H)$. Consequently, either there exists a red subgraph $G$ within $K_N-V(H)$, or there exists a blue subgraph $K_{m-1}$ within $K_N-V(H)$. In the latter case, together with the vertex $x_1$, this forms a blue $K_m$. In both cases, the lemma is proved.
\end{proof}

Combining the bound $r(G,nH)\le r(G,H)+(n-1)|H|$~\cite{Burr1975} with Lemma~\ref{upper}, we obtain the following result.

\begin{corollary}\label{npower}
Let $G$ be a graph with $\ell$ edges and no isolated vertices. For $m\ge 3$, we have
\[
r(G,tK_m)\le (2\ell+1)^{\frac{m-1}{2}}+m(t-1).
\]
\end{corollary}

\section{The trichotomy lemma for sparse graphs}\label{section3}

\begin{proof}[Proof of Lemma~\ref{Gpsize}]
Let $X$ be the set of neighbors of all degree-1 vertices in $G$, $Y$ be the set of vertices with degree at least 3 and not in $X$, and $Z$ be the set of vertices that are neither of degree 1 nor in $X\cup Y$. It is easy to see that all vertices in $Z$ have degree 2 in $G$. Moreover,
\[
|X|\le \ell-1\,,
\]
otherwise, $G$ would contain a matching consisting of at least $\ell$ end-edges, contradicting our assumption.

If $G$ contains a cycle such that all vertices on the cycle belong to $Z$, then considering that $G$ is connected, it follows that $G$ itself is this cycle. Since $G$ contains no suspended path with $q$ vertices, we have $n\le q-1$, which contradicts the assumption. Thus, for any cycle in $G$, at least one vertex on the cycle does not belong to $Z$.

Next, remove all degree-1 vertices from $G$ to obtain a new graph denoted by $G'$. Let $p$ be the number of vertices in $G'$. Then it has $p+k-2$ edges.  Clearly, the vertex set of $G'$ is $X\cup Y\cup Z$, and the vertices in $X$ have degree at least 1 in $G'$, the vertices in $Y$ have degree at least 3 in $G'$, and all vertices in $Z$ have degree 2 in $G'$. In the following, we prove that \[|G'|\le (q-2)(2\ell+3k-8)+1\,.\]

From the total degree of $G'$, it follows that 
\[
|X|+3|Y|+2|Z|\le 2(p+k-2)=2(|X|+|Y|+|Z|+k-2)\,.
\]
Simplifying this inequality yields
\[
|Y|\le |X|+2k-4\le \ell+2k-5\,.
\]

If $G'$ contains a cycle $v_1v_2\cdots v_av_1$ such that $v_2, \ldots,v_a\in Z$, then $v_1\not\in Z$. By the assumption, we have $a\le q-1$. We remove the vertices $v_2,\ldots,v_a$ and their incident edges from $G'$ and add a loop to the vertex $v_1$, where a loop is considered an edge. This operation decreases both the number of vertices and edges by $a-1$. The process is repeated until $G'$ contains no cycles of this form.

If $G'$ contains a suspended path $v_1v_2\cdots v_b$, where $v_2,\ldots,v_{b-1}\in Z$ and $v_1,v_b\in X\cup Y$, then by the assumption, $b\le q-1$. We remove the vertices $v_2,\ldots,v_{b-1}$ and their incident edges from $G'$, and add an edge between $v_1$ and $v_b$. This operation decreases both the number of vertices and edges by $b-2$. The process is repeated until $G'$ contains no vertices from $Z$. The resulting graph is denoted by $G''$.

Note that $G''$ may be a multigraph. In fact, it is constructed by contracting each suspended path into an edge, thereby reducing the same number of vertices and edges. Thus, we have 
\[
|G''|=|X|+|Y|\le 2\ell+2k-6 \quad \text{and} \quad e(G'')=|G''|+k-2\le 2\ell+3k-8\,.
\]

By replacing each edge in $G''$ with a suspended path of appropriate length, we can recover $G'$ from $G''$. Since $G$ does not contain any suspended paths with $q$ vertices, for each loop in $G''$, the reconstruction process adds at most $q-2$ vertices; for each edge in $G''$ that is not a loop, at most $q-3$ vertices are added during the reconstruction. Observing that $G''$ is connected, it contains at least $|G''|-1$ edges that are not loops, and at most $e(G'')-|G''|+1$ edges that are loops. Therefore,
\begin{align*}
    |G'|&\le |G''|+(q-3)(|G''|-1)+(q-2)(e(G'')-|G''|+1) \\
    &= (q-2)e(G'')+1 \\
    &\le (q-2)(2\ell+3k-8)+1\,.
\end{align*}

It follows that the number of vertices in $G$ with degree greater than $1$ is at most $\alpha=(q-2)(2\ell+3k-8)+1$. In other words, the number of vertices in $G$ with degree $1$ is at least $n-\alpha$. Since the total number of neighbors of degree-1 vertices, $|X|$, is at most $\ell-1$, by the pigeonhole principle, $G$ contains a star with at least 
\[
\left\lceil \frac{n-\alpha}{\ell-1} \right\rceil
\]
end-edges. This completes the proof.
\end{proof}

\section{Additional lemmas}\label{section4}

The following two lemmas will be used for the second and third cases of the trichotomy proof, respectively.

\begin{lemma}[Hall~\cite{Hall1935}]\label{Hall}
    Consider a complete bipartite graph $K_{a,b}$, where $a\le b$, with parts $X=\{x_1, x_2, \dots, x_a\}$ and $Y=\{y_1, y_2, \dots, y_b\}$, whose edges are colored red and blue. Then, one of the following holds:
    \begin{enumerate}[(1)]
        \item There exists a red matching of size $a$.
        \item For some $0\le c\le a-1$, there exists a blue subgraph $K_{c+1, b-c}$, where $c+1$ vertices are in $X$.
    \end{enumerate}
\end{lemma}

\begin{lemma}[Luo and Peng~\cite{Luo2023}]\label{star}
	For any $t\ge 1$ and $m\ge 2$, if $n\ge t+2$, then
	\[r(K_{1,n-1},tK_m)=(n-1)(m-1)+t\,.\]
\end{lemma}

To prove a special case of Theorem \ref{main}, we need the following two classical results.

\begin{lemma}[Tur\'an~\cite{Turan1941}]\label{Turan}
	$\operatorname{ex}(n,K_{r+1})\le (1-\frac{1}{r})\frac{n^2}{2}$.
\end{lemma}

\begin{lemma}[Chv\'atal and Harary~\cite{Chvatal1972}]\label{Chvatal}
    For any graph $G$ on $n$ vertices that contains no isolated vertices,
    \[
    r(G,2K_2)=
    \begin{cases}
    n+2, & \text{if } G \text{ is complete}, \\
    n+1, & \text{otherwise}.
    \end{cases}
    \]
\end{lemma}

The following lemma is the special case (when $m=2$) of Theorem \ref{main}.

\begin{lemma}\label{mequals2}
	Let $G$ be a connected graph with $n$ vertices and $n+k-2$ edges. For any positive integer $t$, define $c=\frac{1}{t+2}$ for $t=1,2$, and $c=\frac{1}{4t-5}$ for $t\ge 3$. If $1\le k\le cn^2$ and $n\ge 3\binom{c^{-1}}{2}$, then
	\[
	r(G, tK_2)=n+t-1\,.\]
\end{lemma}

\noindent \textbf{Remark}. We do not seek the optimal value of $c$ here. Through a slightly tedious discussion, it can be shown that when $n$ is sufficiently large, $c$ can be made very close to $\frac{1}{2t-2}$ for $t\ge 2$.

\begin{proof}[Proof of Lemma~\ref{mequals2}]
By Lemma~\ref{lowerbound}, it suffices to prove the upper bound. The conclusion is clearly true when $t=1$. When $t=2$, by Lemma~\ref{Chvatal}, it suffices to show that the graph $G$ is not complete. Since
\[n+k-2\le n+\frac{1}{4}n^2-2<\binom{n}{2}\] for $n\ge 5$, it follows that $G$ is not a complete graph.

When $t\ge 3$, we first prove the following claim.

\begin{claim}\label{K2t-2}
  The graph $\overline{G}$ contains $K_{2t-2}$ as a subgraph.
\end{claim}

\begin{proof}
	\begin{align*}
		e(\overline{G})&=\binom{n}{2}-e(G) \\
		&=\binom{n}{2}-(n+k-2)\\
		&\ge \binom{n}{2}-(n+cn^2-2) \\
		&=\left(1-\frac{2}{4t-5}\right)\frac{n^2}{2}-\frac{3}{2}n+2\\
		&>\left(1-\frac{1}{2t-3}\right)\frac{n^2}{2}.
	\end{align*}
	The final inequality requires only that $n\ge 3(2t-3)(4t-5)=3\binom{c^{-1}}{2}$.

	Now by Lemma~\ref{Turan}, we conclude that $\overline{G}$ contains $K_{2t-2}$ as a subgraph.
\end{proof}

Consider a graph $H$ with $n+t-1$ vertices. We proceed by contradiction, assuming that $H$ does not contain $G$ as a subgraph, and that $\overline{H}$ does not contain $tK_2$ as a subgraph. In $\overline{H}$, let $M$ be a maximum matching with $s$ edges. Then $1\le s\le t-1$. Let $X=V(H)\setminus V(M)$, then $H[X]$ is a complete graph.

We claim that for each edge $uv \in M$, either $u$ or $v$ is nonadjacent to at most one vertex in $X$. If this is not the case, then both $u$ and $v$ are nonadjacent to at least two vertices in $X$. Without loss of generality, suppose $uu', vv' \in E(\overline{H})$, with $u', v' \in X$ and $u' \neq v'$. By removing the edge $uv$ from $M$ and adding edges $uu'$ and $vv'$, we obtain a larger matching in $\overline{H}$, which contradicts the maximality of $M$.

For each edge $uv \in M$, select a vertex from $\{u, v\}$ that has at least $|X|-1$ neighbors in $X$, and add it to the set $W$. Then, we have $|X|+|W|\ge |H|-(t-1)\ge n$. Select any $n-|X|$ vertices from $W$, and denote the resulting set by $W$ as well. Let $H'$ be the subgraph of $H$ induced by $X\cup W$.

Next, we prove that $G$ is a subgraph of $H'$, which leads to the final contradiction. Since $G$ and $H'$ have the same number of vertices, it suffices to show that $\overline{H'}$ is a subgraph of $\overline{G}$. In $\overline{H'}$, at most $2t-2$ vertices have a positive degree, and the remaining vertices are isolated. By Claim~\ref{K2t-2}, $\overline{G}$ contains $K_{2t-2}$ as a subgraph, so it also contains $\overline{H'}$ as a subgraph. This completes the proof of the lemma.
\end{proof}

\section{Proof of Theorem~\ref{main}}
\begin{proof}[Proof of Theorem~\ref{main}]
	Define a constant $c$ that depends only on $m$ and $t$ as follows:
	\[c=c(m,t)=\min\left\{\frac{\varepsilon(m)}{t},\frac{1}{4t}, \frac{1}{3(2q-3)}\right\},\]
	where $q=(m-1)t(mt-1)+2t$, and $\varepsilon=\varepsilon(m)$ is a constant depending solely on $m$, as defined in Theorem~\ref{BEFRS}. We do not aim to maximize $c$ here.

	We employ a double induction on $m$ and $t$ to establish the theorem. For the base cases, when $m=2$, the theorem follows from Lemma~\ref{mequals2}; and when $t=1$, it is verified by Theorem~\ref{BEFRS}. Assuming the theorem holds for $m-1$ and $t-1$, we now proceed to consider the case where $m\ge 3$ and $t\ge 2$.

	To apply the inductive method, it is necessary to ensure that the range of $k$ is also compatible with induction. It is straightforward to verify that
	\[c(m,t)<c(m,t-1) \text{ and}\ c(m,t)<\varepsilon(m)\,.\]
	Furthermore, regardless of the specific values of the constants $c(m,t)$ and $c(m-1,t)$, since $n$ is sufficiently large and $n^{\frac{2}{m-1}}$ is asymptotically negligible compared to $n^{\frac{2}{m-2}}$, it always holds that
	\begin{equation}\label{mtm-1t}
		c(m,t)n^{\frac{2}{m-1}}<c(m-1,t)n^{\frac{2}{m-2}}\,.
	\end{equation}

	By Lemma~\ref{lowerbound}, it suffices to prove the upper bound. Let $N=(n-1)(m-1)+t$. We proceed by contradiction. Assume that $K_N$ is a complete graph with a red-blue edge coloring such that it contains neither a red subgraph isomorphic to $G$ nor a blue subgraph isomorphic to $tK_m$. We will derive a contradiction in the following argument.

	By Lemma~\ref{Gpsize}, we discuss three cases separately.

\begin{case}
	$G$ contains a suspended path of at least $(m-1)t(mt-1)+2t$ vertices.
\end{case}

Let the suspended path be $v_1v_2\cdots v_p$, where $p\ge (m-1)t(mt-1)+2t$. Let $G'$ be the graph obtained from $G$ by shortening the suspended path by $t$ vertices, i.e., replacing the path $v_1v_2\cdots v_p$ with a new path $v_1v_2\cdots v_{p-t-1}v_p$. We make the following claim.

\begin{claim}\label{smallG}
	The graph $K_N$ contains $G'$ as a red subgraph.
\end{claim}

\begin{proof}
	It is easy to see that $1\le k\le c(m,t)n^{\frac{2}{m-1}} < c(m,t-1)n^{\frac{2}{m-1}}$. By the induction hypothesis, we have
	\[r(G,(t-1)K_m)=(n-1)(m-1)+(t-1)\,,\]
	which implies that $K_N$ contains a blue subgraph $(t-1)K_m$. Let $A$ denote the vertex set of this $(t-1)K_m$.

	It is easy to see that $1\le k\le \frac{\varepsilon}{t}n^{\frac{2}{m-1}} < \varepsilon(n-t)^{\frac{2}{m-1}}$. In $K_N-A$, by Theorem~\ref{BEFRS}, we have
	\begin{align*}
		r(G',K_m)&=(n-t-1)(m-1)+1 \\
		&< (n-1)(m-1)+t-(t-1)m\\
		&=N-|A|.
	\end{align*}
	If $K_N-A$ contains a blue $K_m$, then it would combine with the blue $(t-1)K_m$ induced by $K_N[A]$, yielding a blue subgraph $tK_m$, which leads to a contradiction. Thus, $K_N-A$ contains a red subgraph $G'$.
\end{proof}

By Claim~\ref{smallG}, there exists a red subgraph $G'$. Let the set of vertices in $G'$ that are not on the suspended path be denoted by $U$. In the subgraph $K_N-U$, consider a longest red path from vertex $v_1$ to vertex $v_p$ that contains at most $p$ vertices, denoted by $v_1Pv_p$. Such a red path exists by Claim \ref{smallG}. If the path $v_1Pv_p$ contains exactly $p$ vertices, then together with the vertices in $U$, it induces a red subgraph isomorphic to $G$, which is a contradiction. Therefore, the red path $v_1Pv_p$ in the subgraph $K_N-U$ contains at most $p-1$ vertices. Let the set of vertices that are neither in $U$ nor in $v_1Pv_p$ be denoted by $V$. Then we have $|V|\ge (n-1)(m-2)+t$.

From inequality~(\ref{mtm-1t}), we have $1\le k < c(m-1,t)n^{\frac{2}{m-2}}$. By the induction hypothesis,
\[r(G,tK_{m-1})=(n-1)(m-2)+t\,,\] so the subgraph $K_N[V]$ contains a blue subgraph isomorphic to $tK_{m-1}$.

Let the vertices of the path $v_1Pv_p$ be relabeled as $x_1, x_2, \dots, x_a$, where $v_1=x_1$ and $v_p=x_a$. Then, we have $a\ge (m-1)t(mt-1)+t$. Let the vertices of the blue $tK_{m-1}$ be labeled as $y_1, y_2, \dots, y_b$, where $b=(m-1)t$. Also, let $a'=mt$ and $b'=t$. So we have $a\ge b(a'-1)+b'$. By applying Lemma~\ref{pathextension}, since no red path of length $a$ connects $x_1$ to $x_a$, we either have a blue subgraph $K_{mt}$ or there exist $t$ vertices among $x_1, x_2, \dots, x_a$ such that every edge between these $t$ vertices and the vertices $y_1, y_2, \dots, y_b$ is blue. In the first case, we obtain a blue subgraph $K_{mt}$, and in the second case, we obtain a blue subgraph $\overline{K_t}+tK_{m-1}$. In both cases, there exists a blue subgraph $tK_m$, leading to a contradiction.

	\begin{case}
		$G$ contains a matching consisting of at least $2t$ end-edges.
	\end{case}
	
	Since
	\[
	1\le k\le c(m,t)n^{\frac{2}{m-1}} < c(m,t-1)n^{\frac{2}{m-1}}\,,
	\]
	by the induction hypothesis, we have
	\[
	r(G,(t-1)K_m)=(n-1)(m-1)+(t-1)\,,
	\]
	which implies that the graph $K_N$ contains a blue subgraph $(t-1)K_m$. We denote the vertex set of this subgraph by $A$.
	
	From inequality~(\ref{mtm-1t}), we obtain
	\[
	1\le k < c(m-1,t)n^{\frac{2}{m-2}}.
	\]
	In $K_N-A$, by the induction hypothesis, we find
	\begin{align*}
		r(G,tK_{m-1})&=(n-1)(m-2)+t \\
		&< (n-1)(m-1)+t-(t-1)m\\
		&=N-|A|.
	\end{align*}
	Thus, $K_N-A$ contains a blue subgraph $tK_{m-1}$. We denote the vertex set of this subgraph by $B$.
	
	Next, remove a matching consisting of $2t$ end-edges from $G$, and remove the corresponding leaves. The resulting graph is denoted as $G'$. Note that this $G'$ is distinct from the $G'$ used in the previous case.
	
	\begin{claim}\label{claim52}
		The graph $K_N-A-B$ contains a red subgraph isomorphic to $G'$.
	\end{claim}
	
	\begin{proof}
		Since \[
		1\le k\le \frac{\varepsilon}{t}n^{\frac{2}{m-1}} < \varepsilon(n-2t)^{\frac{2}{m-1}},
		\]
		in the graph $K_N-A-B$, by Theorem~\ref{BEFRS}, we have
		\begin{align*}
			r(G',K_m)&=(n-2t-1)(m-1)+1 \\
			&< (n-1)(m-1)+t-(t-1)m-t(m-1)\\
			&=N-|A|-|B|.
		\end{align*}
		If $K_N-A-B$ contains a blue $K_m$, it would form a blue $tK_m$ by combining with the blue $(t-1)K_m$ induced by $A$, which leads to a contradiction. Therefore, $K_N-A-B$ must contain a red subgraph isomorphic to $G'$.
	\end{proof}
	
	By Claim~\ref{claim52}, there exists a red subgraph $G'$ in $K_N-A-B$. In this $G'$, let $X$ denote the set of vertices incident to the matching removed from $G$, and let $Y$ represent the set of vertices not contained in $G'$. We have
	\[
	|X|=2t \quad \text{and} \quad |Y|=N-(n-|X|)=(n-1)(m-2)+3t-1.
	\]
	Note that between every vertex in $X$ and every $K_{m-1}$ in $B$, there is at least one red edge. Otherwise, the subgraph $K_N[X \cup B]$ would contain a blue $K_m$, and combining it with the blue $(t-1)K_m$ from $A$ would yield a blue $tK_m$, leading to a contradiction. Therefore, we claim that every vertex in $X$ is adjacent to at least $t$ vertices in $Y$ with red edges.
	
	By Lemma~\ref{Hall}, there is either a red matching that covers $X$, or a blue $K_{c+1, |Y|-c}$, where $0\le c\le 2t-1$. If the former holds, then $K_N$ would contain a red subgraph isomorphic to $G$, leading to a contradiction. If the latter holds, since every vertex in $X$ is adjacent to at least $t$ vertices in $Y$ with red edges, it follows that $c\ge t$. Thus, there are $t$ vertices in $X$ that are adjacent by blue edges to at least $|Y|-2t+1=(n-1)(m-2)+t$ vertices in $Y$. We denote this set of vertices in $Y$ by $Y'$.
	
	From inequality~(\ref{mtm-1t}), we obtain
	\[
	1\le k < c(m-1,t)n^{\frac{2}{m-2}}\,.
	\]
	In $K_N[Y']$, by the fact that $r(G,tK_{m-1})=(n-1)(m-2)+t$, we conclude that $K_N[Y']$ contains a blue subgraph isomorphic to $tK_{m-1}$. This, together with the vertices in $X$, would induce a blue subgraph $tK_m$, leading to a contradiction.

	\begin{case}
		There does not exist a suspended path in $G$ consisting of $(m-1)t(mt-1)+2t$ vertices, nor does there exist a matching formed by $2t$ end-edges.
	\end{case}
	
	After removing all vertices of degree 1 and their incident edges from the graph $G$, we obtain a graph denoted by $G'$. Clearly, $G'$ is a connected graph, and $e(G')=|G'|+k-2$. 
	By Lemma~\ref{Gpsize}, $|G'|\le \alpha$, where $\alpha=(q-2)(4t+3k-8)+1$ and $q=(m-1)t(mt-1)+2t$. Furthermore, $G$ contains a star formed by $\left\lceil \frac{n-\alpha}{2t-1} \right\rceil$ end-edges, with the center vertex denoted by $v$. By Lemma~\ref{star}, there exists a red star $K_{1,n-1}$ in $K_N$, with the center vertex denoted by $x$. We first embed vertex $v$ into vertex $x$, and then we embed $G'-v$ entirely into the red neighborhood of $x$. To achieve the latter, we need the following claim.
	
	\begin{claim}
		$r(G'-v,tK_m)<n$.
	\end{claim}
	
	\begin{proof}
		Clearly, $e(G'-v)\le e(G')-1\le \alpha+k-3$. By Corollary~\ref{npower}, we have
		\[
		r(G'-v,tK_m)\le\left(2(\alpha+k-3)+1\right)^{\frac{m-1}{2}}+m(t-1).
		\]
		Therefore, we only need to prove that 
		\[
		\left(2(\alpha+k-3)+1\right)^{\frac{m-1}{2}}+m(t-1)<n,
		\]
		which is equivalent to proving
		\[
		2(\alpha+k-3)+1<(n-mt+m)^{\frac{2}{m-1}}.
		\]
		
		Since $k\le\frac{1}{3(2q-3)}n^{\frac{2}{m-1}}$ and $\alpha=(q-2)(4t+3k-8)+1$, we have
		\begin{align*}
			&2(\alpha+k-3)+1 \\
			= & 2\left((q-2)(4t+3k-8)+1+k-3\right)+1 \\
			= & (6q-10)k+2(q-2)(4t-8)-3 \\
			\le & \frac{6q-10}{6q-9}n^{\frac{2}{m-1}}+2(q-2)(4t-8)-3 \\
			< & (n-mt+m)^{\frac{2}{m-1}}.
		\end{align*}		
		The last inequality holds because both $m$ and $t$ are fixed positive integers, and $n$ is sufficiently large. This completes the proof of the claim.
	\end{proof}
	
	By the above claim, we can embed the graph $G'-v$ into the red neighborhood of vertex $x$. Thus, the graph $G'$ has been entirely embedded into the red subgraph of $K_N$.
	
	In the graph $G$, by removing all degree-1 vertices adjacent to $v$, we obtain a graph denoted by $G''$. Clearly, $G''$ contains $G'$ as a subgraph. Furthermore, by Lemma \ref{Gpsize}, 
	\[
	|G''|\le n-\left\lceil \frac{n-\alpha}{2t-1} \right\rceil.
	\]
	Next, based on the embedding of $G'$, we apply a greedy algorithm to embed $G''$ into the red subgraph of $K_N$. If at some step the embedding cannot proceed, there exists a vertex $z$ in $K_N$ that is adjacent to at most $|G''|-2$ red edges. In other words, the number of blue edges incident to vertex $z$ is at least
	\[
	N-|G''|+1\ge (n-1)(m-1)+t-n+\left\lceil \frac{n-\alpha}{2t-1} \right\rceil+1\ge (t-1)m+(n-1)(m-2)+1.
	\]
	The last inequality always holds when $n$ is large.
	
	Since 
	\[
	1\le k\le c(m,t)n^{\frac{2}{m-1}}<c(m,t-1)n^{\frac{2}{m-1}},
	\]
	by the induction hypothesis, we know that
	\[
	r(G,(t-1)K_m)=(n-1)(m-1)+(t-1),
	\]
	which implies that there exists a blue subgraph $(t-1)K_m$ contained in $K_N-z$. We denote the vertex set of this subgraph by $A$. Note that $z\not\in A$.

	Since $z$ is incident to at least $(t-1)m+(n-1)(m-2)+1$ blue edges, in the blue neighborhood of vertex $z$, there are at least $(n-1)(m-2)+1$ vertices that do not belong to the set $A$. Regardless of the values of the constants $c(m,t)$ and $\varepsilon(m-1)$, since $n$ is sufficiently large, $n^{\frac{2}{m-1}}$ is a smaller order term compared to $n^{\frac{2}{m-2}}$, and we always have
	\[
	c(m,t)n^{\frac{2}{m-1}}<\varepsilon(m-1)n^{\frac{2}{m-2}}.
	\]
	By Theorem~\ref{BEFRS}, in the graph induced by these vertices, there exists either a red subgraph isomorphic to $G$, or a blue $K_{m-1}$, which, together with vertex $z$ and the vertices in $A$, induces a blue subgraph containing $tK_m$. Both cases lead to a contradiction.
	
	Thus, $G''$ can always be embedded into the red subgraph of $K_N$. Next, we only need to embed the degree-1 vertices adjacent to $v$, and the entire graph $G$ can be embedded into the red subgraph of $K_N$. This is feasible because the vertex $x$, into which $v$ is embedded, is adjacent to at least $n-1$ red edges, so there are enough vertices to embed the degree-1 neighbors of $v$. This completes the proof of the main theorem.	
\end{proof}

\subsection*{Acknowledgements}

We sincerely thank the two anonymous reviewers for their careful reading of our manuscript and for their valuable comments and suggestions, which have greatly improved the quality of this paper. This research was supported by National Key R\&D Program of China under grant number 2024YFA1013900, NSFC under grant number 12471327, and the Natural Science Foundation of Hebei Province under grant number A2023205045.

\end{document}